\documentclass[a4paper]{article}
\usepackage{amsthm,amsfonts,amsmath,amssymb}
\usepackage[cp1251]{inputenc}
\usepackage[english]{babel}
\usepackage[final]{graphicx}
\usepackage{setspace}
\usepackage[12pt]{extsizes}
\begin{document}
\newcommand{\per}{{\rm per}}
\newtheorem{teorema}{Theorem}
\newtheorem{lemma}{Lemma}
\newtheorem{utv}{Proposition}
\newtheorem{svoistvo}{Property}
\newtheorem{sled}{Corollary}
\newtheorem{con}{Conjecture}

\author{A. A. Taranenko}
\title{Transversals in Completely Reducible Multiary Quasigroups and in Multiary Quasigroups of Order 4}
\date{}

\maketitle

\begin{abstract}
An $n$-ary quasigroup $f$ of order $q$ is an $n$-ary operation over a set of cardinality $q$ such that the Cayley table of the operation is an $n$-dimensional latin hypercube of order $q$. A transversal in a quasigroup $f$ (or in the corresponding latin hypercube) is a collection of $q$ $(n+1)$-tuples from the Cayley table of $f$, each pair of tuples differing at each position. The problem of transversals in latin hypercubes was posed by Wanless in 2011.

An $n$-ary quasigroup $f$ is called reducible if it can be obtained as a composition of two quasigroups whose arity is at least 2, and it is completely reducible if it can be decomposed into binary quasigroups.

In this paper we investigate transversals in reducible quasigroups and in quasigroups of order 4. We find a lower bound on the number of transversals for a vast class of completely reducible quasigroups. Next we prove that, except for the iterated group $\mathbb{Z}_4$ of even arity, every $n$-ary quasigroup of order 4 has a transversal. Also we obtain a lower bound on the number of transversals in quasigroups of order 4 and odd arity and count transversals in the iterated group $\mathbb{Z}_4$ of odd arity and in the iterated group $\mathbb{Z}_2^2.$

All results of this paper can be regarded as those concerning latin hypercubes.
\end{abstract}

\section{Introduction}

Let $\Sigma_q$ be the set $\left\{0, \ldots, q-1\right\}.$ An \textit{$n$-ary quasigroup of order $q$} is a set $\Sigma_q^n$ with an $n$-ary operation $f: \Sigma_q^n \rightarrow \Sigma_q$ such that the equation $x_0 = f(x_1, \ldots, x_n)$ has a unique solution for any one variable if all the other $n$ variables are specified arbitrarily. We identify a quasigroup with its $n$-ary operation $f$. Unless otherwise stated, under a quasigroup we mean a multiary quasigroup. 

A 1-ary quasigroup of order $q$ is a bijection of the set $\Sigma_q$ to itself, i.e. a permutation from the symmetric group $S_q$. A 2-ary quasigroup (a \textit{binary} quasigroup) is a binary operation $*$ over a set $\Sigma_q$ with the following property: for each $a_0, a_1, a_2 \in \Sigma_q$ there exist unique $x_1, x_2 \in \Sigma_q$ such that both $a_0 = a_1 * x_2$ and $a_0 = x_1 * a_2$ hold. The multiplication table of every binary quasigroup of order $q$ is a latin square of order $q$ that is a $q \times q$ table filled by $q$ symbols so that each row and each column contain all different symbols.

In general, if we define an $n$-dimensional latin hypercube of order $q$ as an $n$-dimensional array filled by $q$ symbols so that each line of the hypercube contains distinct symbols, then the multiplication table of an $n$-ary quasigroup of order $q$ is an $n$-dimensional latin hypercube of the same order and vice versa. The formal definition of a latin hypercube will be given later.

A \textit{transversal} in an $n$-ary quasigroup $f$ of order $q$ is a set $\left\{\alpha^i\right\}_{i=1}^{q}$ of $(n+1)$-vectors $\alpha^i = (a_0^i, a_1^i, \ldots, a_n^i),~ a_k^i \in \Sigma_q$ such that $a^i_0 = f (a_1^i, \ldots, a_n^i) $ for all $i \in \left\{1, \ldots, q \right\}$ and $a^i_k \neq a^j_k$ for all $i \neq j$ and $k \in \left\{0, \ldots, n \right\}$. In other words, a quasigroup $f$ has a transversal if there exist permutations $\tau_1, \ldots, \tau_n$ from the symmetric group $S_q$ such that $a = f (\tau_1(a), \ldots, \tau_n(a))$ for all $a \in \Sigma_q.$  We denote by $T(f)$ the number of transversals in a quasigroup $f$.

For $n$-dimensional latin hypercubes of order $q$ we can also define a transversal as a set of $q$ entries containing all distinct symbols and being at the Hamming distance $n$ from each other. There exists a natural bijection between transversals in a quasigroup and in the corresponding latin hypercube, so all problems on transversals in latin hypercubes can be reformulated for quasigroups and vice versa.

Transversals in latin squares were studied in a number of papers. Counting and estimating their number are considered to be rather hard problems. Only in last few years several tight estimates on the number of transversals in large latin squares were proved.   

The problem of an upper bound on the number of transversals in latin squares of order $q$ was posed by Wanless in Loops'03 and soon afterwards the first non-trivial asymptotic bound was proved in~\cite{mckey}. The bound was improved up to $\left((1+o(1))\frac{q}{e^2}\right)^q $ in~\cite{mytrans}, and in~\cite{glebov} it was reproved by another technique. Moreover, in the latter paper it was proposed a probabilistic construction of latin squares that confirmed the exactness of this bound. Also, in~\cite{exactasym} it is found an asymptotic behavior of the number of transversals in the Cayley table of the group $\mathbb{Z}_q$ for odd $q$, which happens to be asymptotically equal to the maximal number of transversals in any latin square.

The existence problem of transversals in latin squares and lower bounds on their numbers are still very far from finalization.  The Ryser's conjecture claiming that every latin square of odd order has a transversal seems to be the most famous and longstanding problem for latin squares. The construction of many latin squares of even order with no transversals was proposed in~\cite{cavwan}.

Our knowledge about transversals in latin hypercubes is even poorer. The trivial upper bound on the number of transversals in $n$-dimensional latin hypercubes of order $q$ is $q!^{n-1}$ that is equal to the number of sets of $q$ $n$-vectors $\left\{(a_1^i, \ldots, a_n^i)\right\}_{i=1}^q$ which do not share a coordinate.  As for latin hypercubes, papers~\cite{glebov} and~\cite{mytrans} bound the number their transversals by $\left((1+o(1))\frac{q^{n-1}}{e^n}\right)^q $ for large $q$. 

In 2011 having analyzed a large amount of computational data, Wanless generalized the Ryser's conjecture and proposed
\begin{con}[\cite{wanless}] \label{hypwan}
Every latin hypercube of odd dimension or odd order has a transversal.
\end{con} 
It is well known that if $n$ and $q$ are both even then the $n$-ary iterated group $\mathbb{Z}_q$ has no transversals (for a proof see, for example,~\cite{obz} or~\cite{wanless}). The existence of transversals in the $n$-ary iterated group $\mathbb{Z}_q$ for odd $n$ was shown for the first time in~\cite{sun}.

The numbers of transversals in all latin hypercubes of orders 2 and 3 are found in~\cite{obz}. The reader is also referred to this paper for some additional results about transversals in latin hypercubes. Since the number of non-equivalent latin hypercubes of order 4 grows rapidly with dimension~\cite{numb4}, order 4 is the first order for which it is hard to  describe all possible numbers of transversals in hypercubes of a given dimension.  

The main aims of this paper are to show the existence of transversals in a majority of $n$-ary completely reducible quasigroups of order $q$ and to prove that among all quasigroups of order 4 only the iterated group $\mathbb{Z}_4$ of even arity has no transversals. In addition, we count the number of transversals in the iterated group $\mathbb{Z}_4$ of odd arity and in the iterated group $\mathbb{Z}^2_2$.  Since the same is true for latin hypercubes, we have a new support for Conjecture~\ref{hypwan}.

\section{Definitions and preliminary results}

Let $n,q \in \mathbb N$, and let $I_q^n= \left\{ (\alpha_1, \ldots , \alpha_n):\alpha_i \in \Sigma_q\right\}$.
An \textit{$n$-dimensional matrix $A$ of order $q$} is an array $(a_\alpha)_{\alpha \in I^n_q}$, $a_\alpha \in\mathbb R$.

Let $m\in \left\{0,\ldots,n\right\}$. An \textit{$m$-dimensional plane} in $A$ is the submatrix of $A$ obtained by fixing $n-m$ indices and letting the other $m$ indices vary from 1 to $n$. The \textit{direction} of a plane is the (0,1)-vector describing which indices are fixed in the plane. An $(n-1)$-dimensional plane is said to be a \textit{hyperplane}, and a $1$-dimensional plane is called a \textit{line}. An $m$-dimensional plane and an $(n-m)$-dimensional plane are \textit{orthogonal} if their directions are orthogonal.

An \textit{$n$-dimensional latin hypercube $Q$ of order $q$} is an $n$-dimensional array of the same order filled by $q$ symbols so that in each line all symbols are different. Let us denote by $Q(f)$ the latin hypercube that is the Cayley table of the quasigroup $f$.

An $n$-ary quasigroup $f$ of order $q$ can be specified by its \textit{graph} 
$$F = \left\{(x_0, x_1, \ldots, x_n) |~ x_0 =f(x_1, \ldots, x_n) \right\}.$$
The set $F$ has a cardinality $q^n$ and consists of entries of the $(n+1)$-dimensional matrix of order $q$ such that the minimal Hamming distance between different entries equals 2. Such subsets of the hypercube are known as \textit{$q$-ary MDS codes of length $n+1$ and with distance $2$}. On the other hand, every such MDS code can be considered as the graph of some quasigroup, so there exists the one-to-one correspondence between $n$-ary quasigroups and MDS codes of length $n+1$. 

An \textit{isotopy} is a collection of $n+1$ permutations $\sigma_i \in S_q, ~ i= 0, \ldots, n$. $n$-ary quasigroups $f$ and $g$ of the same order are called \textit{isotopic}, if for some isotopy $(\sigma_0, \sigma_1, \ldots , \sigma_n)$ we have $$f(x_1, \ldots, x_n) \equiv \sigma^{-1}_0 \left( g(\sigma_1(x_1), \ldots, \sigma_n(x_n)) \right).$$ 

Two $n$-ary quasigroups $f$ and $g$ are said to be \textit{parastrophic} if there exists a permutation $\pi \in S_{n+1}$ such that
$$x_0 = f(x_1, \ldots, x_n) \Leftrightarrow x_{\pi(0)} = g (x_{\pi(1)}, \ldots, x_{\pi(n)}).$$
The permutation $\pi$ is called a \textit{parastrophe}.

The latin hypercubes corresponding to isotopic (parastrophic) quasigroups $f$ and $g$ are also said to be isotopic (parastrophic). If quasigroups $f$ and $g$ can be turned into each other by application of some isotopies and parastrophes, then the latin hypercubes $Q(f)$ and $Q(g)$ belong to the same main class.

It is easy to see that if $\left\{(a_0^i, a_1^i, \ldots, a_n^i)\right\}_{i=1}^q$ is a transversal in a quasigroup $f$, $(\sigma_0, \sigma_1, \ldots , \sigma_n)$ is an isotopy from $f$ to a quasigroup $g$, and $\pi$ is parastrophe from $f$ to another quasigroup $h$, then the sets 
$$\left\{(\sigma_0(a_0^i), \sigma_1(a_1^i), \ldots, \sigma_n(a_n^i)) \right\}_{i=1}^q \mbox{ and }\left\{(a_{\pi(0)}^i, a_{\pi(1)}^i, \ldots, a_{\pi(n)}^i) \right\}_{i=1}^q$$
 are transversals in the quasigroups $g$ and $h$ respectively. So numbers of transversals in isotopic and parastrophic quasigroups are the same. 

An $n$-ary quasigroup $f$ of order $q$ is a \textit{composition} of an $(n-m)$-ary quasigroup $h$ and an $(m+1)$-ary quasigroup $g$  if there exists a permutation $\sigma \in S_{n+1}$ such that for all $x_1, \ldots, x_n \in \Sigma_q$
$$f(x_1, \ldots , x_n) = x_0  \Leftrightarrow  g(x_{\sigma(0)}, x_{\sigma(1)}, \ldots , x_{\sigma(m)}) = h(x_{\sigma(m+1)}, \ldots , x_{\sigma(n)}).$$
It means that there exists a direction of $m$-dimensional planes in the latin hypercube $Q(f)$ such that all planes of this direction are the Cayley tables of the quasigroups $g(a, x_{\sigma(1)}, \ldots , x_{\sigma(m)}) = y_0$ for $a \in \Sigma_q$ and in some orthogonal $(n-m)$-dimensional plane we have the latin hypercube $Q(h)$.

A quasigroup $f$ is \textit{permutably reducible} if it is a composition of two quasigroups, each of them having arity at least 2. Further we waive the word "permutably". 
 
An $n$-ary quasigroup $f$ with $n \geq 3$ is \textit{completely reducible} if it is a composition of completely reducible quasigroups $h_1$ and $h_2$ having arity at least 2. Meanwhile all binary quasigroups are also considered to be completely reducible. 

One of the simplest examples of completely reducible quasigroups is the  \textit{$n$-ary iterated group $G$}:
$$f(x_1, \ldots, x_n) = x_0 \Leftrightarrow x_0 + x_1 + \ldots + x_n = 0,~ x_i \in G, $$
where $+$ means an operation of $G$.

For future convenience we need the following property of completely reducible quasigroups.

\begin{lemma}\label{pred}
Let $f$ be an $n$-ary completely reducible quasigroup of order $q$ that defines an MDS code $F$.  Then for some permutation $\pi \in S_{n+1}$ there exists the quasigroup $g$ corresponding to the same MDS code $F$ and defined by the equation
$$h_1(x_{\pi(0)}, x_{\pi(1)}, \ldots, x_{\pi(n-2)}) = h_2 ( x_{\pi(n-1)}, x_{\pi (n)}),$$
where $h_1$ is a completely reducible $(n-1)$-ary quasigroup of order $q$ and $h_2$ is a binary quasigroup of order $q$.
\end{lemma}

In other words, if $f$ is an $n$-ary completely reducible quasigroup, then there exists a direction of lines in the latin hypercube $Q(f)$ such that there are exactly $n$ different lines of this direction (given by the quasigroup $h_2$) and all orthogonal hyperplanes are the Cayley tables of isotopic completely reducible quasigroups generated by the quasigroup $h_1$.  

We will say that the quasigroup $g$ defined in Lemma~\ref{pred} is a \textit{proper representation} of a quasigroup $f$, and the quasigroup $h_2$ is the \textit{external} quasigroup for the proper representation $g$. Note that a completely reducible $n$-ary quasigroup  may have several proper representations. A proof of this lemma uses techniques of the quasigroup theory and can be found in Appendix.

Let us consider quasigroups of order 4 now. Throughout  the paper we widely use the functions $l : \left\{0,1,2,3\right\} \rightarrow \mathbb{Z}_2$ and $\nu : \left\{0,1,2,3\right\} \rightarrow \left\{0,1,2,3\right\}$ (or $\nu : \mathbb{Z}_2 \rightarrow \mathbb{Z}_2$ depending the context) that are defined as
$$l(0)=l(1)=0,~ l(2)=l(3)=1;$$
$$\nu(0) = 1,~ \nu(1) =0, ~ \nu(2) = 3,~ \nu(3) =2.$$
These functions always act coordinate-wise to vectors.

To separate a special class of quasigroups of order 4 we need the concept of a Boolean function. A function $\lambda: \mathbb{Z}_2^n \rightarrow \mathbb{Z}_2$ is said to be a  \textit{Boolean function}, and the set $\mathbb{Z}_2^n$ is known as the \textit{$n$-dimensional Boolean hypercube}.  For a Boolean vector $z \in \mathbb{Z}_2^n,~ z = (z_1, \ldots, z_n)$ the \textit{weight} of $z$ is
$$w(z) = z_1 + \ldots + z_n.$$

We say that an $n$-ary quasigroup $f$ of order 4 is \textit{standardly semilinear} if there exists a Boolean function $\lambda: \mathbb{Z}_2^n \rightarrow \mathbb{Z}_2$ such that
\begin{gather*}
x_0 = f(x_1, \ldots, x_n) \Leftrightarrow \\ \Leftrightarrow l(x_0) \oplus \ldots \oplus l(x_n) = 0 \mbox{ and } x_0 \oplus \ldots \oplus x_n \oplus \lambda(l(x_1), \ldots, l(x_n)) = 0,
\end{gather*}
where $\oplus$ means modulo 2 addition. Note that the function $\lambda$ uniquely determines the standardly semilinear quasigroup $f$ and vice versa.

A quasigroup $f$ is called \textit{semilinear} if it is isotopic to some standardly semilinear quasigroup.
Standard semilinearity of an $n$-ary quasigroup $f$ means that the corresponding $n$-dimensional latin hypercube $Q(f)$ is a disjoint union of $2^n$ latin subhypercubes of order $2$ (in case $n=2$ latin subsquares of order 2 are known as intercalates~\cite{wanless}). Position occupied by each subhypercube we call \textit{block}. For all standardly semilinear quasigroups $f$ the same blocks in $Q(f)$ are filled by the same symbols. The orientation function $\lambda$ specifies one of two possible latin hypercubes for each block. For convenience, we will say that a latin hypercube $Q(f)$ of a standardly semilinear quasigroup $f$ is \textit{intercalated}.

The following characterization of $n$-ary quasigroups of order 4 was obtained in~\cite{krpt}:

\begin{teorema} \label{krpt}
Every $n$-ary quasigroup of order $4$ is reducible or semilinear.
\end{teorema}

Let us introduce two important examples of quasigroups of order 4.
An $n$-ary  quasigroup $f$ of order 4 is called \textit{$\mathbb{Z}_4$-linear} (\textit{$\mathbb{Z}_2^2$-linear}) if it is isotopic to the iterated group $\mathbb{Z}_4$ (to the iterated group $\mathbb{Z}_2^2$). In the following two lemmas we list main properties of $\mathbb{Z}_4$-linear and $\mathbb{Z}_2^2$-linear quasigroups. Their proofs are given in Appendix.

\begin{lemma} \label{z4char}
Suppose $f$ is an $n$-ary $\mathbb{Z}_4$-linear quasigroup.
\begin{enumerate}
\item $f$ is a completely reducible quasigroup.
\item $f$ is semilinear and is isotopic to the standardly semilinear quasigroup $h_{\mathbb{Z}_4}$ defined by the Boolean function $\lambda_{\mathbb{Z}_4}$ such that $\lambda_{\mathbb{Z}_4}(z) =0$ if weight of $z$ equals $0$ or $3$ by modulo $4$ and $\lambda_{\mathbb{Z}_4}(z) =1$ if weight of $z$ equals $1$ or $2$ by modulo $4$.
\item A standardly semilinear quasigroup defined by a Boolean function $\lambda$ is $\mathbb{Z}_4$-linear if and only if for every $2$-dimensional plane $P$ in the $n$-dimensional Boolean hypercube it holds $\sum\limits_{z \in P} \lambda(z) \equiv 1 \mod 2.$ 
\end{enumerate}
\end{lemma} 

\begin{lemma} \label{z22char}
Suppose $f$ is an $n$-ary $\mathbb{Z}_2^2$-linear quasigroup.
\begin{enumerate}
\item $f$ is a completely reducible quasigroup.
\item $f$ is  semilinear  and is isotopic to the standardly semilinear quasigroup $h_{\mathbb{Z}_2^2}$ defined by the Boolean function $\lambda_{\mathbb{Z}_2^2}$ such that $\lambda_{\mathbb{Z}_2^2}(z) =0$ for all $z$.
\item A standardly semilinear quasigroup defined by a Boolean function $\lambda$ is $\mathbb{Z}_2^2$-linear if and only if for every $2$-dimensional plane $P$ in the $n$-dimensional Boolean hypercube it holds $\sum\limits_{z \in P} \lambda(z) \equiv 0 \mod 2.$ 
\end{enumerate}
\end{lemma}

\section{Main results}

For most of the proofs in this papers, it is more convenient to use current quasigroup terminology instead to introduce equivalent one for latin hypercubes. So all results will also be given in terms of quasigroups. In this section we provide their equivalents for latin hypercubes.

The first result concerns the number of transversals in the Cayley tables of completely reducible quasigroups.

\begin{teorema} \label{compredlatin}
Let $Q(f)$ be an $n$-dimensional latin hypercube of order $q$ which is the Cayley table  of a completely reducible $n$-ary quasigroup $f$. 
\begin{enumerate}
\item If $n$ is odd then $Q(f)$ has at least $(q\cdot q!)^{\frac{n-1}{2}}$ transversals. 
\item If $n$ is even and there exists a direction of $2$-dimensional planes such that all latin squares of this direction have a transversal and are isotopic to the Cayley table of an external quasigroup for one of a proper representations of $f$, then $Q(f)$ has at least $(q\cdot q!)^{\left\lfloor  \frac{n-1}{2} \right\rfloor}$ transversals.
\end{enumerate}
\end{teorema}     

Next result is a lower bound on the numbers of transversals in latin hypercubes of odd dimension  being the Cayley tables of standardly semilinear quasigroups. 

\begin{teorema} \label{lowsemioddlatin}
Let $Q$ be an $n$-dimensional intercalated latin hypercube of order $4$ where $n$ is odd. Then $Q$ has at least $\frac{1}{3} (16^{n-1} + 2 \cdot 8^{n-1})$ transversals.
\end{teorema}

For an arbitrary latin hypercube of order $4$ and odd dimension we prove a weaker lower bound on the number of transversals.

\begin{teorema} \label{oddnlatin}
Let $Q$ be an $n$-dimensional latin hypercube of order $4$ where $n$ is odd. Then $Q$ has at least $8^{n-1}$ transversals.
\end{teorema}

Also we obtain a characterization of latin hypercubes of order 4 with no transversals.

\begin{teorema} \label{harzerolatin}
Let $Q$ be an $n$-dimensional latin hypercube of order $4$ without transversals. Then $n$ is even and $Q$ is isotopic to the Cayley table of the iterated group $\mathbb{Z}_4$.
\end{teorema}

At last we count transversals in the Cayley tables of the iterated groups $\mathbb{Z}_4$ and $\mathbb{Z}_2^2.$

\begin{teorema} \label{numtranslatin}
Let $Q^n(\mathbb{Z}_4)$ (or $Q^n(\mathbb{Z}_2^2)$) be the $n$-dimensional latin hypercube  isotopic to the Cayley table of the iterated group $\mathbb{Z}_4$ (or the iterated group $\mathbb{Z}_2^2$, respectively).
\begin{enumerate}
\item If $n$ is odd then the number of transversals in the latin hypercube $Q^n(\mathbb{Z}_2^2)$ is equal to the number of transversals in $Q^n(\mathbb{Z}_4)$ and equals $\frac{3}{8} \cdot 24^{n-1} + 5 \cdot 8^{n-2}$.
\item If $n$ is even then the latin hypercube $Q^n(\mathbb{Z}_4)$ has no transversals and the number of transversals in $Q^n(\mathbb{Z}_2^2)$ is  $\frac{3}{8} \cdot 24^{n-1} -  8^{n-2}$. 
\end{enumerate} 
Moreover,  these latin hypercubes are unique by isotopy and parastrophe having the maximum number of transversals among all intercalated latin hypercubes.
\end{teorema}

Note that the trivial upper bound on the number of transversals in an $n$-dimensional latin hypercube of order 4 is $4!^{n-1} = 24^{n-1}$ and the numbers of transversals in latin hypercubes $Q^n(\mathbb{Z}_2^2)$ and $Q^n(\mathbb{Z}_4)$ of odd dimension are quite close to this bound. So we propose
\begin{con}
Among all latin hypercubes of order $4$, a maximum number of transversals is contained in the odd dimensional Cayley tables of the iterated group $\mathbb{Z}_4$ and in the Cayley tables of the iterated group $\mathbb{Z}_2^2$. 
\end{con}
Computational data from~\cite{wancec} confirm this conjecture for all $n \leq 5$.

\section{Transversals in completely reducible quasigroups}

In this section we prove a lower bound on the number of transversals in a majority of $n$-ary completely reducible quasigroups that was announced in~\cite{obz}.  For this purpose we need the following two lemmas that were initially obtained in the same paper. These lemmas will be widely used through the paper, and since they are simple to prove, we repeat their proofs here. The meaning of both lemmas for latin hypercubes of order $q$ is that we can construct a transversal by choosing transversally $q$ $k$-dimensional planes and  taking then by one element from each of chosen planes. 

Throughout the lemmas we suppose that an $n$-ary quasigroup $f$ of order $q$ is a composition of an $(n-m)$-ary quasigroup $h$ and an $(m+1)$-ary quasigroup $g$:
$$f(x_1, \ldots , x_n) = x_0  \Leftrightarrow g(x_{\sigma(0)}, x_{\sigma(1)}, \ldots , x_{\sigma(m)}) = h(x_{\sigma(m+1)}, \ldots , x_{\sigma(n)})$$
for some permutation $\sigma \in S_{n+1}.$

\begin{lemma} \label{quasi1}
If a quasigroup $f$ is a composition of quasigroups $g$ and $h$ having $T(g)$ and $T(h)$ transversals respectively, then $f$ has at least $T(g)T(h)$ transversals:
$$T(f) \geq T(g)T(h).$$
\end{lemma}

\begin{proof}
Let $\left\{ (a^i_{m+1}, \ldots, a_{n}^i, i)\right\}_{i=1}^q$ be a transversal in the quasigroup $h$  and $\left\{ ( a_0^i ,a_1^i, \ldots, a_m^i, i)\right\}_{i=1}^q$ be a transversal in the quasigroup $g$. It can be checked that the set $$\left\{ (a_{\sigma^{-1}(0)}^i, a_{\sigma^{-1}(1)}^i, \ldots, a_{\sigma^{-1}(n)}^i)\right\}_{i=1}^q$$
is a transversal in the quasigroup $f$, and different pairs of transversals from $g$ and $h$ give distinct transversals in $f$. 
\end{proof}

\begin{lemma} \label{quasi2}
Assume that for some $a \in \Sigma_q$ the $(n-m -1)$-ary quasigroup $h^a$ defined by the equation $h( x_{m+1}, \ldots, x_n) = a$ has $T(h^a)$ transversals and the $m$-ary quasigroup $g^a$ defined by the equation $g(x_0, x_1, \ldots, x_m) = a$ has $T(g^a)$ transversals. Then the quasigroup $f$ has at least $q! \cdot T(h^a)T(g^a)$ transversals:
$$T(f) \geq q! \cdot T(h^a)T(g^a).$$
\end{lemma}

\begin{proof}
Let $\left\{ ( a^i_{m+1},  \ldots, a_{n}^i)\right\}_{i=1}^q$ be a transversal in the quasigroup $h^a$ and  $\left\{ (a_0^j, a_1^j, \ldots, a_{m}^j)\right\}_{j=1}^q$ be a transversal in $g^a$. It is easy to check that for every permutation $\tau \in S_q$ the set 
$$\left\{ (a_{\sigma^{-1}(0)}^{\tau(i)}, a_{\sigma^{-1}(1)}^{\tau(i)},  \ldots, a_{\sigma^{-1}(m)}^{\tau(i)}, a_{\sigma^{-1}(m+1)}^{i}, \ldots, a_{\sigma^{-1}(n)}^{i})\right\}_{i=1}^q$$
is a transversal in the quasigroup $f$ and each pair of transversals from $g^a$ and $h^a$ produces $q!$ distinct transversals in $f$.
\end{proof}

Using these lemmas, we now prove a lower bound on the number of transversals in certain completely reducible quasigroups. As a corollary, we obtain Theorem~\ref{compredlatin}.

\begin{teorema}\label{compred}
Let $f$ be a completely reducible $n$-ary quasigroup of order $q$. 
\begin{enumerate}
\item If $n$ is odd then $f$ has at least $(q\cdot q!)^{\frac{n-1}{2}}$ transversals. 
\item If $n$ is even and the external quasigroup for one of proper representations of $f$ has a transversal, then $f$ has at least $(q\cdot q!)^{\left\lfloor  \frac{n-1}{2} \right\rfloor}$ transversals.
\end{enumerate}
\end{teorema}

\begin{proof}
1. The proof is by induction on $n$. It is easy to see that a permutation (a 1-ary quasigroup) has a unique transversal. 

Since an $n$-ary quasigroup $f$ is completely reducible, we can apply Lemma~\ref{pred} and instead of the quasigroup $f$ we examine a quasigroup $g$ defined by the equation
$$ h_1(x_{0}, x_1, \ldots, x_{n-2})  = h_2(x_{n-1}, x_{n}),$$ 
where $h_1$ is a completely reducible $(n-1)$-ary quasigroup and $h_2$ is a binary quasigroup.

For each $a \in \Sigma_q$ consider the 1-ary quasigroup $h_2^a$ defined by the equation $h_2(x_{n-1}, x_n) = a$ and the $(n-2)$-ary quasigroup $h_1^a$ defined by the equation $h_1(x_0, x_1, \ldots, x_{n-2}) =a.$ Since there are $q$ ways to choose $a$, Lemma~\ref{quasi2} implies
$$T(g) \geq (q \cdot q!)T(h_1^a)T(h_2^a) = (q \cdot q!)T(h_1^a).$$
Note that $n-2$ is odd, so by the inductive assumption, the quasigroup $h_1^a$ has at least $(q \cdot q!)^{\frac{n-3}{2}}$ transversals. Since quasigroups $f$ and $g$ have the same number of transversals, we obtain $T(f) \geq (q\cdot q!)^{\frac{n-1}{2}}$.

2. Under made assumptions and by Lemma~\ref{pred}, for the quasigroup $f$ there exists a proper representation $g$ defined by the equation
$$ h_1(x_{0}, x_1, \ldots, x_{n-2})  = h_2(x_{n-1}, x_{n}),$$ 
where $h_1$ is a completely reducible $(n-1)$-ary quasigroup and the binary quasigroup $h_2$ has a transversal.
By the previous clause, the quasigroup $h_1$ has at least $(q \cdot q!)^{\frac{n-2}{2}}$ transversals. Using Lemma~\ref{quasi1}, we obtain 
$$T(f) = T(g) \geq T(h_1)T(h_2) \geq (q\cdot q!)^{\left\lfloor  \frac{n-1}{2} \right\rfloor}.$$
\end{proof}

\section{Transversals in semilinear quasigroups}

Due to a handy definition of standardly semilinear quasigroups, a set of all their transversals can be divided into describable parts that allows us to analyze the numbers of transversals in semilinear quasigroups. Let us introduce the main tools serving this purpose.

Further, under a quadruple we will mean a multiset of 4 Boolean vectors. 
A quadruple  $\left\{z^1, z^2, z^3, z^4\right\}$, where $z^j$ are Boolean $(n+1)$-vectors,  is \textit{proper} if for all $i \in \left\{0, \ldots, n\right\}$ the set $\left\{z^1_i, z^2_i, z^3_i, z^4_i\right\} $ coincides with the set $ \left\{0, 0, 1, 1\right\}$ as a multiset.  In other words, a quadruple  is proper if at each position it covers zeroes and ones exactly twice.

In most cases we will be interested in a special subset of proper quadruples, namely worthwhile quadruples. A proper quadruple  $\left\{z^1, z^2, z^3, z^4\right\}$ is \textit{worthwhile} if each $z^j$ has an even weight.

Recall that an $n$-ary quasigroup $f$ of order 4 is standardly semilinear if there exists a Boolean function $\lambda: \mathbb{Z}_2^n \rightarrow \mathbb{Z}_2$ such that
\begin{gather*}
x_0 = f(x_1, \ldots, x_n) \Leftrightarrow \\ \Leftrightarrow l(x_0) \oplus \ldots \oplus l(x_n) = 0 \mbox{ and } x_0 \oplus \ldots \oplus x_n \oplus \lambda(l(x_1), \ldots, l(x_n)) = 0,
\end{gather*}
where the function $l$ is defined as $l(0)=l(1)=0,~ l(2)=l(3)=1.$ The following lemma makes clear the worth of worthwhile quadruples.

\begin{lemma} \label{worth}
Let $f$ be a standardly semilinear $n$-ary quasigroup of order $4$.
Suppose $\left\{\alpha^1, \alpha^2, \alpha^3, \alpha^4\right\}$ is a transversal in the quasigroup $f$. Then the quadruple $\left\{l(\alpha^1), l(\alpha^2), l(\alpha^3), l(\alpha^4)\right\}$ is worthwhile. 
\end{lemma}
In other words, if $Q$ is an intercalated latin hypercube, then each transversal in $Q$ belongs to blocks whose positions may be presented only as a worthwhile quadruple in the Boolean hypercube.
The lemma trivially follows from the definitions of a transversal, of a standardly semilinear quasigroup, and of a worthwhile quadruple.

Next we divide all worthwhile quadruples into two classes, which we call twin and brindled quadruples. Note that if a worthwhile quadruple $\left\{z^1, z^2, z^3, z^4\right\}$ contains two identical $(n+1)$-vectors, the other two vectors are the same. Moreover, in this case we may assume that $z^1 = z^2$ and $z^3 = z^4 = \nu(z^1) = \nu(z^2)$. We will say that worthwhile quadruples composed of two pairs of identical vectors are \textit{twin} quadruples and worthwhile quadruples formed by four different vectors are \textit{brindled}. It is easy to see that twin quadruples exist only if $n+1$ is even (that is $n$ is odd), and brindled quadruples exist for all $n \geq 2$.

For an intercalated latin hypercube $Q$, a twin quadruple defines a pair of diagonally located blocks and a brindled quadruples gives four different blocks such that each hyperplane of $Q$ intersects exactly two blocks.

We start with the investigation of transversals that can be given by twin quadruples. 

\begin{lemma} \label{twintrans}
Let $n$ be odd and let $f$ be a standardly semilinear $n$-ary quasigroup of order $4$ such that
\begin{gather*}
x_0 = f(x_1, \ldots, x_n) \Leftrightarrow \\ \Leftrightarrow l(x_0) \oplus \ldots \oplus l(x_n) = 0 \mbox{ and } x_0 \oplus \ldots \oplus x_n \oplus \lambda(l(x_1), \ldots, l(x_n)) = 0.
\end{gather*}
The number of transversals $\left\{\alpha^1, \alpha^2, \alpha^3, \alpha^4\right\}$ for which  $\left\{l(\alpha^1), l(\alpha^2), l(\alpha^3), l(\alpha^4)\right\}$ is a twin quadruple is equal to $8^{n-1}$.
\end{lemma} 

\begin{proof}

Note that the number of twin quadruples is equal to the number of unordered pairs $\left\{z, \nu(z)\right\}$ from the Boolean $(n+1)$-dimensional hypercube such that both $z$ and $\nu(z)$ have an even weight. Consequently, there are exactly $2^{n-1}$ twin quadruples. 

Let $\left\{z, z, \nu(z), \nu(z)\right\}$ be a twin quadruple. Consider all unordered pairs $\left\{\alpha^1, \alpha^2\right\}$ of $(n+1)$-vectors, where $\alpha^1 = (a_0^1, \ldots, a_n^1)$, $\alpha^2 = (a_0^2, \ldots, a_n^2)$, $a^i_j \in \Sigma_4$, satisfying the following conditions:
\begin{itemize} 
\item $l(a_0^1) = l(a_0^2) = z_0, \ldots, l(a_n^1) = l(a_n^2) = z_n;$
\item $a_0^1 \neq a_0^2, \ldots, a_n^1 \neq a_n^2;$
\item $a_0^1 \oplus \ldots \oplus a_n^1 = a_0^2 \oplus \ldots \oplus a_n^2 = \lambda(z_1, \ldots, z_n)$.
\end{itemize}

For every $(n+1)$-vector $\alpha^1$ satisfying these conditions there exists a unique complement $\alpha^2$, because the conditions implies $\alpha^2 = \nu(\alpha^1)$. Consequently, there are $2^{n-1}$ different pairs  $\left\{\alpha^1,  \alpha^2\right\}$ that will be the first two elements of transversals.

The last two elements $\alpha^3$ and $\alpha^4$ of transversals in the quasigroup $f$ we construct independently of $\alpha^1$ and $\alpha^2$ in a similar way but using $\nu(z)$ instead $z$. It is easy to see that all $\alpha^i,~i=1, \ldots, 4$, differ at all positions, and so they compose a transversal. Also every transversal $\left\{\beta^1, \beta^2, \beta^3, \beta^4\right\}$ for which  $\left\{l(\beta^1), l(\beta^2), l(\beta^3), l(\beta^4)\right\} = \left\{z, z, \nu(z), \nu(z)\right\}$ can be obtained by this construction. Therefore, twin quadruples produce exactly $\left(2^{n-1}\right)^3$ transversals in the quasigroup $f$.
\end{proof}

For an $n$-dimensional itercalated latin hypercube $Q$,  Lemma~\ref{twintrans} may be rewritten and proved in a simpler way. In this lemma we count transversals that belong to pairs of diagonally located blocks. If $n$ is even then diagonal blocks contain latin subhypercubes over the same set of symbols, and so such pairs of blocks can not produce transversals. If $n$ is odd then we independently choose a transversal in each latin subhypercube, unite them and obtain a transversal in $Q$. It only remains to multiply the number of constructed transversals by the number of diagonally located blocks.  

From Lemma~\ref{twintrans} a lower bound on the number of transversals in semilinear quasigroups of odd arity follows.

\begin{sled}\label{semiodd}
Let $f$ be a semilinear $n$-ary quasigroup of order $4$ where $n$ is odd. Then $f$ has at least $8^{n-1}$ transversals.
\end{sled}
 
Consider now how many transversals each brindled quadruple  can generate in a standardly semilinear quasigroup. In further statements for an $(n+1)$-vector $z = (z_0, z_1, \ldots, z_n)$ we denote by $\overline{z}$ the $n$-vector $(z_1, \ldots, z_n)$.

\begin{lemma} \label{brindtrans}
Let $f$ be a standardly semilinear $n$-ary quasigroup such that
\begin{gather*}
x_0 = f(x_1, \ldots, x_n) \Leftrightarrow \\ \Leftrightarrow l(x_0) \oplus \ldots \oplus l(x_n) = 0 \mbox{ and } x_0 \oplus \ldots \oplus x_n \oplus \lambda(l(x_1), \ldots, l(x_n)) = 0.
\end{gather*}
Suppose $\left\{z^1, z^2, z^3, z^4\right\}$ is a brindled quadruple of Boolean $(n+1)$-vectors.
\begin{enumerate}
\item If $\lambda(\overline{z^1}) \oplus \lambda(\overline{z^2}) \oplus \lambda(\overline{z^3}) \oplus \lambda(\overline{z^4}) = 1$ then in the quasigroup $f$ there are no transversals $\left\{\alpha^1, \alpha^2, \alpha^3, \alpha^4\right\}$ such that $\left\{l(\alpha^1), l(\alpha^2), l(\alpha^3), l(\alpha^4)\right\} = \left\{z^1, z^2, z^3, z^4\right\}.$
\item If $\lambda(\overline{z^1}) \oplus \lambda(\overline{z^2}) \oplus \lambda(\overline{z^3}) \oplus \lambda(\overline{z^4}) = 0$ then there exist exactly $2 \cdot 4^{n-1}$ transversals $\left\{\alpha^1, \alpha^2, \alpha^3, \alpha^4\right\}$ such that $\left\{l(\alpha^1), l(\alpha^2), l(\alpha^3), l(\alpha^4)\right\} = \left\{z^1, z^2, z^3, z^4\right\}.$
\end{enumerate}
\end{lemma}

This lemma means that for an intercalated latin hypercube $Q$  the number of transversals in  four properly located blocks depends on orientations of latin subhypercubes  in the blocks and does not depend on arrangements of the blocks. Moreover, given the orientation function $\lambda$ and a quadruple of blocks, we can find the number of their transversals.

\begin{proof}
1. Assume that  the quasigroup $f$ has a transversal $\left\{\alpha^1, \alpha^2, \alpha^3, \alpha^4\right\}$ for which $\lambda(l(\overline{\alpha^1})) \oplus \lambda(l(\overline{\alpha^2})) \oplus \lambda(l(\overline{\alpha^3})) \oplus \lambda(l(\overline{\alpha^4})) = 1$.
Consider the sum $$S = \sum\limits_{j=1}^4 a^j_0 \oplus \ldots \oplus a^j_n \oplus \lambda(l(\overline{\alpha^j})).$$
$S$ equals zero because, by the definition of the quasigroup $f$, each term does. On the other hand, equalities $a^1_i \oplus \ldots \oplus a^4_i = 0$ for all $i \in \left\{0,\ldots, n\right\}$ and $\sum\limits_{j=1}^4 \lambda(l(\overline{\alpha^j})) \equiv  1 \mod 2$ imply that 
$$S = \lambda(l(\overline{\alpha^1})) \oplus \lambda(l(\overline{\alpha^2})) \oplus \lambda(l(\overline{\alpha^3})) \oplus \lambda(l(\overline{\alpha^4})) = 1;$$
 a contradiction.

2. We firstly prove that for every function $\mu: \left\{1,2, 3, 4\right\} \rightarrow \mathbb{Z}_2$ such that $\mu(1) \oplus \mu(2) \oplus \mu(3) \oplus \mu(4) = 0$
and for any proper quadruple $\left\{z^1, z^2, z^3, z^4\right\}$ of distinct $(n+1)$-vectors there exist exactly $2 \cdot 4^{n-1}$ proper quadruples $\left\{y^1, y^2, y^3, y^4\right\}$ of Boolean $(n+1)$-vectors satisfying the following conditions:
\begin{description}
\item[(1)] If $z^j_i = z^k_i$ then $y^j_i \neq y^k_i$.
\item[(2)] $y_0^j \oplus y_1^j \oplus \ldots \oplus y^j_n \oplus \mu(j) = 0$ for all $j \in \left\{1, 2, 3, 4 \right\}$.
\end{description}

Hereafter in the lemma we suppose all quadruples to be ordered. Vectors $y^j$ will be used to choose by one element of a future transversal from the blocks of a brindled  quadruple.

If $n=1$ then without loss of generality we may assume that 
$$z^1 = (0,0), ~ z^2= (0,1),  ~ z^3 = (1,0), ~ z^4 = (1,1).$$ 
It is easy to see that for each function $\mu$ there exist exactly two proper quadruples $\left\{y^1, y^2, y^3, y^4\right\}$ satisfying (1),(2). Let us list all such quadruples $\left\{y^1, y^2, y^3, y^4\right\}$ for all main types of the function $\mu$:

\begin{enumerate}
\item $\mu(1) = \mu(2) = \mu(3) = \mu(4) = 0$.
$$y^1 = (0,0), ~ y^2 = (1,1),~  y^3 = (1,1),~ y^4 = (0,0).$$
$$y^1 = (1,1), ~ y^2 = (0,0),~  y^3 = (0,0),~ y^4 = (1,1).$$
\item $\mu(1) = \mu(2) = 0, ~ \mu(3) = \mu(4) = 1$.
$$y^1= (1,1), ~ y^2 = (0,0), ~ y^3 = (1,0),~ y^4 = (0,1).$$
$$y^1 = (0,0), ~ y^2 = (1,1),~  y^3 = (0,1),~ y^4 = (1,0).$$
\item $\mu(1) = \mu(4) = 0, ~ \mu(2) = \mu(3) = 1$.
$$y^1 = (0,0), ~ y^2 = (1,0),~ y^3 = (0,1),~ y^4 = (1,1).$$
$$y^1 = (1,1), ~ y^2 = (0,1),~  y^3 = (1,0),~ y^4 = (0,0).$$
\item $\mu(1) = \mu(2) = \mu(3) = \mu(4) = 1$.
$$y^1 = (0,1), ~ y^2 = (1,0),~ y^3 = (1,0),~ y^4 = (0,1).$$
$$y^1 = (1,0), ~ y^2 = (0,1),~  y^3 = (0,1),~ y^4 = (1,0).$$
\end{enumerate}

If $n \geq 2$ then the number of proper quadruples $\left\{y^1, y^2, y^3, y^4\right\}$ satisfying (1) and (2) also does not depend on the function $\mu$. Indeed, let $\mu$ and $\mu'$ be different functions such that 
$$\mu(1) \oplus \mu(2) \oplus \mu(3) \oplus \mu(4) = \mu'(1) \oplus \mu'(2) \oplus \mu'(3) \oplus \mu'(4) = 0.$$
Assume that $z_{i_1}^k \neq z_{i_2}^k$, $z_{i_1}^k \neq z_{i_1}^j$, and $z_{i_1}^j \neq z_{i_2}^j$. Such indices $i_1, i_2, j,$ and $k$ exist because vectors $z^1, z^2, z^3,$ and $z^4$ are all different. By the case $n=1$, every proper quadruple  $\left\{y^1, y^2, y^3, y^4\right\}$ satisfying (2) with the function $\mu$ can be turned in two ways to a proper quadruple  $\left\{y'^1, y'^2, y'^3, y'^4\right\}$ satisfying (2) with the function $\mu'$ by changing values of all vectors in positions $i_1$ and $i_2$. Note that there are 8 different functions $\mu: \left\{1,2, 3, 4\right\} \rightarrow \mathbb{Z}_2$ such that $\mu(1) \oplus \mu(2) \oplus \mu(3) \oplus \mu(4) = 0$.

Since for a given proper quadruple $\left\{z^1, z^2, z^3, z^4\right\}$ of $(n+1)$-vectors there exist exactly $4^{n+1}$ proper quadruples $\left\{u^1, u^2, u^3, u^4\right\}$ for which  $z^j_i = z^k_i$ implies $u^j_i \neq u^k_i$, we obtain $2 \cdot 4^{n-1}$ proper quadruples $\left\{y^1, y^2, y^3, y^4\right\}$ satisfying (1) and (2) for a quadruple $\left\{z^1, z^2. z^3, z^4\right\}$ and a function $\mu$.

Let now $\left\{z^1, z^2, z^3, z^4\right\}$ be a brindled quadruple of Boolean $(n+1)$-vectors such that $\lambda(\overline{z^1}) \oplus \lambda(\overline{z^2}) \oplus \lambda(\overline{z^3}) \oplus \lambda(\overline{z^4}) = 0$ and $\left\{y^1, y^2, y^3,y^4\right\}$ be a proper quadruple satisfying (1),(2) with $\mu(j) = \lambda(\overline{z^j})$.

Put
$$a^j_i = 2 z^j_i + y^j_i$$
for all $j \in \left\{1, 2, 3, 4\right\}$ and $i \in \left\{0, \ldots, n\right\}$, where $z^j_i$ and $y^j_i$ are Boolean, but addition and multiplication are in $\mathbb{Z}$, and put $\alpha^j = (a_0^j, a_1^j, \ldots, a_n^j)$.

It can be shown by direct calculations that conditions (1),(2) guarantee that $\left\{\alpha^1, \alpha^2, \alpha^3, \alpha^4\right\}$ is a transversal in the quasigroup $f$. For a given transversal $\left\{\beta^1, \beta^2, \beta^3, \beta^4\right\}$ corresponding to the brindled quadruple $\left\{z^1, z^2, z^3, z^4\right\}$ we uniquely determine the quadruple $\left\{y^1, y^2, y^3, y^4\right\}$ satisfying these conditions. Therefore, the number of transversals generated by the brindled quadruple $\left\{z^1, z^2, z^3, z^4\right\}$ is equal to $2 \cdot 4^{n-1}$.
\end{proof}

As a corollary of this lemma we obtain the following characterization for the orientation function $\lambda$ of an intercalated latin hypercube $Q$ with no transversals.

\begin{utv}\label{crit}
Let $n$ be even and let $f$ be a standardly semilinear $n$-ary quasigroup of order $4$ such that 
\begin{gather*}
x_0 = f(x_1, \ldots, x_n) \Leftrightarrow \\ \Leftrightarrow l(x_0) \oplus \ldots \oplus l(x_n) = 0 \mbox{ and } x_0 \oplus \ldots \oplus x_n \oplus \lambda(l(x_1), \ldots, l(x_n)) = 0.
\end{gather*}

The quasigroup $f$ has no transversals if and only if for each brindled quadruple $\left\{z^1, z^2, z^3, z^4\right\}$ it holds $\lambda(\overline{z^1}) \oplus \lambda(\overline{z^2}) \oplus \lambda(\overline{z^3}) \oplus \lambda(\overline{z^4}) = 1.$ 
\end{utv}

For a deeper insight on transversals in semilinear quasigroups we need to know a behavior of Boolean functions on brindled quadruples. The following lemma serves exactly this purpose. For the sequel we state the lemma in the most general form. We omit its proof here but give it in Appendix.

 \begin{lemma} \label{bool}
Let $\lambda$ be a Boolean function on the $n$-dimensional Boolean hypercube. Suppose that for every brindled quadruple $\left\{z^1, z^2, z^3, z^4\right\}$ of $(n+1)$-vectors it holds $\lambda(\overline{z^1}) \oplus \lambda(\overline{z^2}) \oplus \lambda(\overline{z^3}) \oplus \lambda(\overline{z^4}) = \delta.$
\begin{itemize}
\item If $\delta = 1$ and $n$ is even, then for every $2$-dimensional plane $P$ in the $n$-dimensional Boolean hypercube we have $\sum\limits_{z \in P} \lambda(z) \equiv 1 \mod 2.$ If $n$ is odd then such Boolean function $\lambda$ does not exist. Moreover, for any odd $n$ and for any Boolean function $\lambda$ the sum $\lambda(\overline{z^1}) \oplus \lambda(\overline{z^2}) \oplus \lambda(\overline{z^3}) \oplus \lambda(\overline{z^4})$ equals to $0$ for at least $\frac{1}{6}(4^{n-1} -2^{n-1})$ brindled quadruples. 
\item If $\delta = 0$ and $n$ is even then for every $2$-dimensional plane $P$ in the $n$-dimensional Boolean hypercube it holds $\sum\limits_{z \in P} \lambda(z) \equiv 0 \mod 2.$  If $n$ is odd then for every $2$-dimensional planes $P_1$ and $P_2$ we have $\sum\limits_{z \in P_1} \lambda(z) \equiv \sum\limits_{z \in P_2} \lambda(z) \mod 2.$
\end{itemize}
\end{lemma}

As a corollary of this lemma and Lemmas~\ref{twintrans} and~\ref{brindtrans} we obtain a more exact lower bound on the number of transversals in semilinear quasigroups of odd arity, and consequently obtain Theorem~\ref{lowsemioddlatin}.
\begin{sled}
Let $f$ be a semilinear $n$-ary quasigroup of order $4$ where $n$ is odd. Then $f$ has at least $\frac{1}{3} (16^{n-1} + 2 \cdot 8^{n-1})$ transversals.
\end{sled}

We are also ready to state the following result.

\begin{utv}\label{semi}
An $n$-ary semilinear quasigroup $f$ of order $4$ has no transversals if and only if $n$ is even and $f$ is a $\mathbb{Z}_4$-linear quasigroup.
\end{utv}

\begin{proof}
Corollary~\ref{semiodd} implies that $n$ is even. The uniqueness of $\mathbb{Z}_4$-linear quasigroups as quasigroups without transversals among all semilinear quasigroups of even arity follows from Proposition~\ref{crit}, Lemma~\ref{bool}, and Lemma~\ref{z4char}.
\end{proof}

\section{Transversals in multiary quasigroups of order 4}

Having the lower bound on the number of transversals in semilinear quasigroups of odd arity and the characterization of semilinear quasigroups of even arity without transversals, we are ready to prove the similar results for general multiary quasigroups of order 4. We start with a lower bound on the number of transversals in quasigroups of order 4 and odd arity. As a corollary of this bound we have Theorem~\ref{oddnlatin}.

\begin{teorema} \label{oddn}
Let $f$ be an $n$-ary quasigroup of order $4$ where $n$ is odd. Then $f$ has at least $8^{n-1}$ transversals.
\end{teorema}

\begin{proof}
The proof is by induction on $n$. When $n=1$ there is nothing to prove.

Let $n$ be odd greater than 1. By Theorem~\ref{krpt}, every $n$-ary quasigroup of order 4 is semilinear or reducible. If $f$ is a semilinear quasigroup then, by Corollary~\ref{semiodd}, it has at least $8^{n-1}$ transversals.

If $f$ is a reducible quasigroup then for $1 \leq m \leq n-2$ there exist an $(n-m)$-ary quasigroup $h$, an $(m+1)$-ary quasigroup $g$, and a permutation $\sigma \in S_{n+1}$ such that 
$$f(x_1, \ldots , x_n) = x_0 \Leftrightarrow g(x_{\sigma(0)}, x_{\sigma(1)}, \ldots , x_{\sigma(m)}) = h(x_{\sigma(m+1)}, \ldots , x_{\sigma(n)}).$$

If $m+1$ and $n-m$ are both odd, then by the inductive assumption, the quasigroups $h$ and $g$ have at least $8^{n-m-1}$ and $8^{m}$ transversals respectively. By Lemma~\ref{quasi1},
$$T(f) \geq T(h)T(g) \geq 8^{n-1}.$$

If $m+1$ and $n-m$ are both even, then for every $a \in \Sigma_4$ we consider the $(n-m-1)$-ary quasigroup $h^a$  defined by the equation $h(x_{m+1}, \ldots, x_n) = a$ and the $m$-ary quasigroup $g^a$ defined by the equation $g(x_0, x_1, \ldots, x_m) = a$. By the inductive assumption, the quasigroups $h^a$ and $g^a$ have at least $8^{n-m-2}$ and $8^{m-1}$ transversals respectively. Therefore Lemma~\ref{quasi2} implies
$$T(f) \geq 96 \cdot T(h^a) T(g^a) \geq 96 \cdot 8^{n-3} > 8^{n-1}.$$
\end{proof}

To prove the next theorem we need one more lemma and several new concepts. An isotopy $(\sigma_0, \sigma_1, \ldots, \sigma_n)$ between $n$-ary quasigroups $f_1$ and $f_2$ is called \textit{principal} if $\sigma_0$ is the identical permutation. 
We will say that quasigroups $f_1$ and $f_2$ are \textit{principally isotopic} if there exists a principle isotopy between them. In other words, quasigroups $f_1$ and $f_2$ are principally isotopic if and only if
$$f_1(x_1, \ldots, x_n) \equiv f_2(\sigma_1(x_1), \ldots, \sigma_n(x_n)).$$
Note that an existence of a principal isotopy between quasigroups $f_1$ and $f_2$ is equivalent to that latin hypercubes $Q(f_1)$ and $Q(f_2)$ can be turned to each other by permutations of hyperplanes but without permutations on symbols.

Since there are isotopic quasigroups  of order 4 that are not principally isotopic, we divide a set of isotopic quasigroups into \textit{principal classes} that are closed under principle isotopies.  For example, all binary $\mathbb{Z}^2_2$-linear quasigroups are in the same principal class but there exist binary $\mathbb{Z}_4$-linear quasigroups belonging different principal classes.  

The notion of principal classes of isotopic quasigroups is needed to mark off $n$-ary $\mathbb{Z}_4$-linear quasigroups from other quasigroups being composed of binary $\mathbb{Z}_4$-linear quasigroups.  For example, consider the following two 3-dimensional latin hypercubes of order $4$ represented by layers:

$$  \begin{array} {cccc}
0 & 1 & 2 & 3 \\ 1 & 0 & 3 & 2 \\ 2 & 3 & 1 & 0 \\ 3 & 2 & 0 & 1
\end{array}
\times
\begin{array} {cccc}
1 & 0 & 3 & 2 \\ 0 & 1 & 2 & 3 \\ 3 & 2 & 0 & 1 \\ 2 & 3 & 1 & 0
\end{array}
\times
\begin{array} {cccc}
2 & 3 & 1 & 0 \\ 3 & 2 & 0 & 1 \\ 1 & 0 & 3 & 2 \\ 0 & 1 & 2 & 3
\end{array}
\times
\begin{array} {cccc}
3 & 2 & 0 & 1 \\ 2 & 3 & 1 & 0 \\ 0 & 1 & 2 & 3 \\ 1 & 0 & 3 & 2
\end{array}
$$
and
$$ \begin{array} {cccc}
0 & 1 & 2 & 3 \\ 1 & 0 & 3 & 2 \\ 2 & 3 & 1 & 0 \\ 3 & 2 & 0 & 1
\end{array}
\times
\begin{array} {cccc}
1 & 0 & 3 & 2 \\ 2 & 3 & 1 & 0 \\ 3 & 2 & 0 & 1 \\ 0 & 1 & 2 & 3
\end{array}
\times
\begin{array} {cccc}
2 & 3 & 1 & 0 \\ 3 & 2 & 0 & 1 \\ 0 & 1 & 2 & 3 \\ 1 & 0 & 3 & 2
\end{array}
\times
\begin{array} {cccc}
3 & 2 & 0 & 1 \\ 0 & 1 & 2 & 3 \\ 1 & 0 & 3 & 2 \\ 2 & 3 & 1 & 0
\end{array}.
$$

Both latin hypercubes are the Cayley tables of quasigroups that are compositions of two binary $\mathbb{Z}_4$-linear quasigroups, but they belong to different main classes and have different numbers of transversals. The former is the Cayley table of a 3-ary $\mathbb{Z}_4$-linear quasigroup, and the latter is not.

From the definitions it is easy to see that an $n$-ary quasigroup with $n\geq 3$ defined by the equation
$$((\ldots((x_1 + x_2) +' x_3) \ldots +^{(n-2)} x_{n-1}) +^{(n-1)} x_n) = x_0,$$
where $+^{(i)}$ are operations of some binary quasigroups, is $\mathbb{Z}_4$-linear only if all of these operations define quasigroups that belong to the same principle class of binary $\mathbb{Z}_4$-linear quasigroups.

Let us state the following auxiliary lemma.

\begin{lemma} \label{almostz4}
Let $f$ be an $n$-ary quasigroup of order $4$.
Suppose for all $a \in \Sigma_4$  the quasigroups defined by the equations $f(x_1, \ldots, x_n) = a$ are $\mathbb{Z}_4$-linear. Then the quasigroup $f$ is isotopic to 
$$(x_1 + \ldots + x_{n-1}) +' x_n  = x_0 \mbox{ or } (x_1 + \ldots + x_{n-1}) \oplus x_{n} = x_0,$$
where $+$ and $+'$ are operations in (possibly the same) principal classes of binary $\mathbb{Z}_4$-linear quasigroups, and $\oplus$ is a $\mathbb{Z}_2^2$-addition. 
\end{lemma}

For a proof of the lemma the reader is referred to Appendix. Further for a binary operation $+$ defining a quasigroup of order 4 we denote by $+^{-1}$ its \textit{right inverse}:
$$x_1 + x_2 = x_0 \Leftrightarrow x_1 = x_0 +^{-1} x_2.$$
It is easy to see that $\oplus^{-1} = \oplus$.

Now we are ready to prove the characterization of $n$-ary quasigroups of order 4 having no transversals which is equivalent to Theorem~\ref{harzerolatin}.

\begin{teorema} \label{harzero}
Let $f$ be an $n$-ary quasigroup of order $4$ without transversals. Then $n$ is even and $f$ is a $\mathbb{Z}_4$-linear quasigroup.
\end{teorema}

\begin{proof}
The proof is by induction on $n$.  Up to equivalence there exist only two binary quasigroups of order 4: the $\mathbb{Z}_4$-linear quasigroup having no transversals and the $\mathbb{Z}_2^2$-linear quasigroup containing 8 transversals.
If $n$ is odd then Theorem~\ref{oddn} implies that an $n$-ary quasigroup of order 4 has a transversal.

Let $n \geq 4$  be even. Suppose that for all $k < n$ among all $k$-ary quasigroups of order 4 only $\mathbb{Z}_4$-linear quasigroups of even arity have no transversals. By Theorem~\ref{krpt}, every $n$-ary quasigroup of order 4 is reducible or semilinear.  For semilinear quasigroups the statement is true by Proposition~\ref{semi}. So we may assume that a quasigroup $f$ with no transversals is reducible that is for some $(m+1)$-ary quasigroup $g$, $(n-m)$-ary quasigroup $h$ (where $1 \leq m \leq n-2$), and some permutation $\sigma \in S_{n+1}$ it holds
$$f(x_1, \ldots , x_n) = x_0 \Leftrightarrow g(x_{\sigma(0)}, x_{\sigma(1)}, \ldots , x_{\sigma(m)}) = h(x_{\sigma(m+1)}, \ldots , x_{\sigma(n)}).$$

For definiteness assume that $m+1$ is odd and $n-m$ is even (the case of even $m+1$ and odd $n-m$ is analogical). By Theorem~\ref{oddn}, the quasigroup $g$ has transversals, so by Lemma~\ref{quasi1}, the quasigroup $f$  have no transversals only if the quasigroup $h$ has no transversals. The inductive hypothesis implies that $h$ is a $\mathbb{Z}_4$-linear quasigroup and so $h$ is isotopic to the quasigroup defined by the equation $x_{\sigma(m+1)} + \ldots +  x_{\sigma(n)}  = y_0,$ where $+$ is an operation from one of principle classes of $\mathbb{Z}_4$-linear quasigroups.

Next for each $a \in \mathbb{Z}_4$ we consider the quasigroups $h^a$ and $g^a$ defined by the equations $h(x_{\sigma(m+1)}, \ldots , x_{\sigma(n)}) = a$ and $g(x_{\sigma(0)}, x_{\sigma(1)}, \ldots , x_{\sigma(m)}) = a$ respectively. Now the quasigroups $h^a$ have odd arity and the quasigroups $g^a$ have even arity, so by the inductive hypothesis and by Lemma~\ref{quasi2}, the quasigroup $f$ cannot have transversals only if for each $a \in \mathbb{Z}_4$ the quasigroups $g^a$ are $\mathbb{Z}_4$-linear.

By Lemma~\ref{almostz4}, the quasigroup $g$ is isotopic to 
$$(x_{\sigma(0)} +' \ldots +' x_{\sigma(m-1)}) +'' x_{\sigma(m)}  = y_0 \mbox{ or to } (x_{\sigma(0)} +' \ldots +' x_{\sigma(m-1)}) \oplus x_{\sigma(m)} = y_0,$$
where $+'$ and $+''$ are operations in principal classes of binary $\mathbb{Z}_4$-linear quasigroups, and $\oplus$ is a $\mathbb{Z}_2^2$-addition. 

Thus, the quasigroup $f$ is isotopic to the quasigroup defined by the equation
$$(x_{\sigma(0)} +' \ldots +' x_{\sigma(m-1)}) +'' x_{\sigma(m)}  = x_{\sigma(m+1)} + \ldots +  x_{\sigma(n)}$$  or to the quasigroup $$(x_{\sigma(0)} +' \ldots +' x_{\sigma(m-1)}) \oplus x_{\sigma(m)} = x_{\sigma(m+1)} + \ldots +  x_{\sigma(n)}.$$

In the last case we have that the quasigroup $f$ is isotopic to the quasigroup 
$$x_{\sigma(0)} +' \ldots +' x_{\sigma(m-2)}  = ((x_{\sigma(m+1)} + \ldots +  x_{\sigma(n)}) \oplus x_{\sigma(m)}) +'^{-1} x_{\sigma(m-1)}$$
that has a transversal as a composition of a quasigroup of odd arity and a non-$\mathbb{Z}_4$-linear quasigroup of even arity.

Let us analyze the first possibility for the quasigroup $f$. If the binary operations $+'$ and $+''$ or the operations $+'$ and $+^{-1}$ define different principle classes then consider  the quasigroup defined by the equation
$$((x_{\sigma(0)} +' \ldots +' x_{\sigma(m-1)}) +'' x_{\sigma(m)}) +^{-1}  x_{\sigma(n)}  = x_{\sigma(m+1)} + \ldots +  x_{\sigma(n-1)}$$
that is isotopic to the quasigroup $f$.
This quasigroup is a composition of a non-$\mathbb{Z}_4$-linear quasigroup of even arity  and a quasigroup of odd arity, therefore by Lemma~\ref{quasi1}, the quasigroup $f$ has at least one transversal.

If the binary operations $+'$, $+''$, and $+^{-1}$ define the same principle class of $\mathbb{Z}_4$-linear quasigroups then the quasigroup $f$ is the $n$-ary $\mathbb{Z}_4$-linear quasigroup defined by the equation
$$x_{\sigma(0)} +' \ldots +' x_{\sigma(m)} +' x_{\sigma(n)} +' \ldots +' x_{\sigma(m+2)} = x_{\sigma(m+1)},$$
that has no transversals.
\end{proof}

\section{Transversals in the iterated groups $\mathbb{Z}_2^2$ and $\mathbb{Z}_4$ and some computational results}

In this section we count the number of transversals in $\mathbb{Z}_4$-linear quasigroups of odd arity  and in all $n$-ary $\mathbb{Z}^2_2$-linear quasigroups. Also we compare theoretical bounds on the numbers of transversals obtained in the previous sections with computational data for quasigroups of small order and arity.

Recall that by Lemma~\ref{worth}, every transversal in a standardly semilinear quasigroup give a twin or brindled quadruple. The number of transversals corresponding to twin quadruples was counted in Lemma~\ref{twintrans}. To find how many transversals can correspond to brindled quadruples we need to know the number of brindled quadruples (the number of quadruples of blocks in intercalated latin hypercubes that can contain transversals).  

\begin{lemma} \label{numbrind}
Let $W(n)$ be the number of brindled quadruples in the Boolean $(n+1)$-dimensional hypercube. Then
\begin{itemize}
\item $W(n) = \frac{1}{32} \left(6^n - 2^n\right)$ if $n$ is even;
\item $W(n) = \frac{1}{32} \left(6^n - 3 \cdot 2^n\right)$ if $n$ is odd.
\end{itemize}
\end{lemma}

The proof of this lemma is technical and is given in Appendix. Let us prove now the following result that is equivalent to Theorem~\ref{numtranslatin}. 

\begin{teorema} \label{numtrans}
\begin{enumerate}
\item If $n$ is odd then the number of transversals in an $n$-ary $\mathbb{Z}_2^2$-linear quasigroup is equal to the number of transversals in an $n$-ary $\mathbb{Z}_4$-linear quasigroup and equals $\frac{3}{8} \cdot 24^{n-1} + 5 \cdot 8^{n-2}$.
\item If $n$ is even then the number of transversals in an $n$-ary $\mathbb{Z}_2^2$-linear quasigroup is  $\frac{3}{8} \cdot 24^{n-1} -  8^{n-2}$. 
\end{enumerate} 
Moreover,  these $n$-ary quasigroups  are unique by isotopy and parastrophe having the maximal number of transversals among all semilinear quasigroups.
\end{teorema}

\begin{proof}
1. By Lemma~\ref{twintrans},  every semilinear quasigroup of odd arity $n$ contains $8^{n-1}$ transversals obtained from twin quadruples. 
 Also recall that by Lemma~\ref{brindtrans}, in a standardly semilinear quasigroup defined by a Boolean function $\lambda$ a brindled quadruple $\left\{z^1, z^2, z^3, z^4\right\}$ gives transversals  if and only if $\lambda(\overline{z^1}) \oplus \lambda(\overline{z^2}) \oplus \lambda(\overline{z^3}) \oplus \lambda(\overline{z^4}) = 0$.

By Lemma~\ref{z22char}, every $n$-ary $\mathbb{Z}_2^2$-linear quasigroup is isotopic to the standardly semilinear quasigroup $f_{\mathbb{Z}_2^2}$ defined by the identical zero Boolean function $\lambda_{\mathbb{Z}_2^2}$. Consequently, for every brindled quadruple $\left\{z^1, z^2, z^3, z^4\right\}$ the sum    $\lambda_{\mathbb{Z}_2^2}(\overline{z^1}) \oplus \lambda_{\mathbb{Z}_2^2}(\overline{z^2}) \oplus \lambda_{\mathbb{Z}_2^2}(\overline{z^3}) \oplus \lambda_{\mathbb{Z}_2^2}(\overline{z^4}) $ equals zero, and so every brindled quadruple produces transversals in the quasigroup $f_{\mathbb{Z}_2^2}$.

Next by Lemma~\ref{z4char}, every $n$-ary $\mathbb{Z}_4$-linear quasigroup is isotopic to the standardly semilinear quasigroup $f_{\mathbb{Z}_4}$ defined the Boolean function $\lambda_{\mathbb{Z}_4}$ such that $\lambda_{\mathbb{Z}_4} (z) = 0$ if weight of $z$ is congruent to 0 or 3 by modulo 4, and $\lambda_{\mathbb{Z}_4} (z) = 1$ if weight of $z$ is congruent to 1 or 2 by modulo 4.

Let $\left\{z^1, z^2, z^3, z^4\right\}$ be a brindled quadruple in the $(n+1)$-dimensional Boolean hypercube. By definition, each $z^{j}$ has an even weight and $w(z^1) + w(z^2) + w(z^3)+ w(z^4) = 2n +2 \equiv 0 \mod 4.$ Then the number of $z^{j}$ having weight congruent to 2 by modulo 4 is even.   Consequently an even number of $\overline{z^{j}}$ have weight congruent to 1 or 2 by modulo 4.  Therefore,
$$\lambda_{\mathbb{Z}_4}(\overline{z^1}) \oplus \lambda_{\mathbb{Z}_4}(\overline{z^2}) \oplus \lambda_{\mathbb{Z}_4}(\overline{z^3}) \oplus \lambda_{\mathbb{Z}_4}(\overline{z^4}) = 0,$$
and every brindled quadruple produces transversals in the quasigroup $f_{\mathbb{Z}_4}$.

By Lemma~\ref{brindtrans}, each brindled quadruple generates exactly $2 \cdot 4^{n-1}$ transversals. Thus,
\begin{gather*}
T(f_{\mathbb{Z}_2^2}) = T(f_{\mathbb{Z}_4}) = 2 \cdot 4^{n-1} W(n) + 8^{n-1} = \\ = \frac{1}{16} \cdot 4^{n-1} \left(6^n - 3 \cdot 2^n\right) + 8^{n-1} =  \frac{3}{8} \cdot 24^{n-1} + 5 \cdot 8^{n-2}.  
\end{gather*}

2. If $n$ is even then every transversal in a semilinear $n$-ary quasigroup gives only brindled quadruples.

As before,  an $n$-ary $\mathbb{Z}_2^2$-linear quasigroup  is isotopic to the standardly semilinear quasigroup $f_{\mathbb{Z}_2^2}$ defined by the identical zero Boolean function and  every brindled quadruple produces exactly $2 \cdot 4^{n-1}$ transversals. Therefore,

$$T(f) = 2\cdot 4^{n-1} W(n) = \frac{1}{16} \cdot 4^{n-1} \left(6^n -  2^n\right) =  \frac{3}{8} \cdot 24^{n-1} -  8^{n-2}.$$

Since for the quasigroups $f_{\mathbb{Z}_4}$ of odd arity and $f_{\mathbb{Z}_2^2}$  every brindled quadruple gives transversals, these quasigroups have the maximal number of transversals among all semilinear quasigroups. By Lemma~\ref{bool}, if for some Boolean function $\lambda$ and for each brindled quadruple $\left\{z^1, z^2, z^3, z^4\right\}$ it holds $\lambda(\overline{z^1}) \oplus \lambda(\overline{z^2}) \oplus \lambda(\overline{z^3}) \oplus \lambda(\overline{z^4}) = 0$, then for every 2-dimensional plane $P$ in the $n$-dimensional Boolean hypercube the sum $\sum\limits_{z \in P} \lambda(z)$ is congruent to either 0 or 1  by modulo 2 for odd $n$ and congruent to 0 for even $n$. By Lemmas~\ref{z4char} and~\ref{z22char}, such Boolean functions $\lambda$ can define only $\mathbb{Z}_4$-linear and $\mathbb{Z}_2^2$-linear quasigroups. 
\end{proof}

In conclusion, we compare lower bounds on numbers of transversals in quasigroups of small arity and order with the minimum values obtained with the help of~\cite{wancec} and in personal communication with I. M. Wanless.

\begin{itemize}

\item $n = 3.$
\begin{itemize}
\item $q=4:$ The minimum number of transversals is equal to 96, that is greater than the lower bound 64 from Theorem~\ref{oddn} but coincides with the bound for completely reducible quasigroups from Theorem~\ref{compred}.
\item $q=5:$ The minimum number of transversals is 859, and by Theorem~\ref{compred}, all completely reducible quasigroups have at least 600 transversals. 
\item $q=6:$ The minimum number of transversals is 7632, and the number of transversals in completely reducible quasigroups is greater than 4320.
\end{itemize}

\item $n=4$.
\begin{itemize}
\item $q=4:$ There exists a unique quasigroup up to equivalence without transversals.
\item $q=5:$ The minimum number of transversals is 60843, and by Theorem~\ref{compred}, the number of transversals in all completely reducible quasigroups is greater than 600.
\end{itemize}

\item $n=5$.
\begin{itemize}
\item $q=4:$  The minimum number of transversals is equal to 18432, that is greater than the lower bound 4096 from Theorem~\ref{oddn} and the bound 9216 for completely reducible quasigroups from Theorem~\ref{compred}.
\item $q=5:$ The minimum number of transversals is 8096923, and the number of transversals in completely reducible quasigroups is greater than 360000.
\end{itemize}
\end{itemize} 

Obtained data allows us to suggest that for quasigroups of order 4 and odd arity the lower bound from Theorem~\ref{oddn} is close to the actual minimum number of transversals.

\section{Appendix}

\textbf{Proof of Lemma~\ref{pred}}

\begin{proof} 
It is known that the structure of every completely reducible $n$-ary quasigroup $f$ can be presented as an unrooted binary tree $\mathcal{T}$, in which $n-1$ vertices of degree 3 (that are called inner vertices) correspond to binary operations  composing $f$, and leaves are labeled by variables $x_0, x_1, \ldots, x_n.$ The tree $\mathcal{T}$ uniquely determines the MDS code $F$ for the quasigroup $f$. More detailed description and additional properties of the tree $\mathcal{T}$ can be found, for example, in~\cite{pot}.

Since the number of leaves of the tree $\mathcal{T}$ is greater than the number of inner vertices, there exists an inner vertex adjacent to 2 leaves. Assuming that this vertex corresponds to the binary quasigroup $h_2$, the tree $\mathcal{T}$ defines the quasigroup
$$g(x_{\pi(1)}, \ldots, x_{\pi(n)}) = x_{\pi(0)} \Leftrightarrow h_1(x_{\pi(0)}, x_{\pi(1)}, \ldots, x_{\pi(n-2)} ) = h_2(x_{\pi(n-1)},  x_{\pi(n)}),$$
where $\pi \in S_{n+1}$ is some permutation and $h_1$ is a completely reducible $(n-1)$-ary quasigroup.
\end{proof}

\textbf{Proof of Lemma~\ref{z4char}}

\begin{proof}
1. By definition, the iterated group $\mathbb{Z}_4$ is completely reducible and an application of an isotopy preserves this property.  

2. Let us prove that the $n$-ary iterated group $\mathbb{Z}_4$ defined by the equation
$$x_0 + x_1 + \ldots + x_n = 0 ,~ x_i \in \mathbb{Z}_4$$
is isotopic to the standardly semilinear quasigroup $h_{\mathbb{Z}_4}$ defined by the Boolean function $\lambda_{\mathbb{Z}_4}$.

We claim that the isotopy $(\sigma, \sigma, \ldots, \sigma)$, where
$$\sigma(0) = 0, ~ \sigma(1) = 2, ~ \sigma(2) = 1, ~\sigma(3) = 3,$$
turns the iterated group $\mathbb{Z}_4$ into the standardly semilinear quasigroup $h_{\mathbb{Z}_4}.$

Indeed, it can be checked that for all $a_1, a_2, a_3 \in \mathbb{Z}_4$
$$a_1 + a_2 = a_3 \Rightarrow l(\sigma(a_1)) + l(\sigma(a_2)) = l(\sigma(a_3)) \mod 2.$$
Therefore for each $(y_0, y_1, \ldots, y_n)$ such that $y_0 = h_{\mathbb{Z}_4}(y_1, \ldots, y_n)$ we have $l(y_0) \oplus \ldots \oplus l(y_n) = 0$, so the obtained via such isotopy quasigroup $h_{\mathbb{Z}_4}$ is standardly semilinear.

Let us prove that this quasigroup is defined by the Boolean function $\lambda_{\mathbb{Z}_4}$ such that $\lambda_{\mathbb{Z}_4}(z) =0 $ if   $w(z) \equiv 0$ or $3 \mod 4$ and $\lambda_{\mathbb{Z}_4}(z) =1 $ if $w(z) \equiv 1$ or $2 \mod 4$. 

Consider an $(n+1)$-vector $(y_0, y_1, \ldots, y_n)$ from the graph of the quasigroup $h_{\mathbb{Z}_4}$ and assume that $w(l(y_1), \ldots, l(y_n)) \equiv k \mod 4.$ Since $l(y_0) \oplus l(y_1) \oplus \ldots \oplus l(y_n) = 0$, we have $l(y_0) \equiv k \mod 2.$  So $w(l(y_0), l(y_1), \ldots, l(y_n)) \equiv 2 \mod 4$ if $k$ is equal to 1 or 2, and $w(l(y_0), l(y_1), \ldots, l(y_n)) \equiv 0 \mod 4$ if $k$ is equal to 0 or 3. 

Let us denote by $N'_a$ the number of indices $i$ such that $y_{i} = a,$ and by $N_a$ the number of indices $i$ such that $\sigma^{-1}(y_{i}) = a.$ By the definition of permutation $\sigma$, we have
$$N_0 = N'_0; ~ N_1 = N'_2; ~ N_2 = N'_1; ~ N_3= N'_3.$$ 

We consider the case when $k$ is equal to 1 or 2 and $w(l(y_0), l(y_1), \ldots, l(y_n)) \equiv 2 \mod 4$. For  another case the reasoning is similar.
Note that the equality $w(l(y_0), l(y_1), \ldots, l(y_n)) \equiv 2 \mod 4$ means $N_2' + N_3' \equiv 2 \mod 4 $. Then $N_1 + N_3 \equiv 2 \mod 4. $ Also, because of
$\sigma^{-1}(y_0) + \sigma^{-1}(y_1) + \ldots + \sigma^{-1}(y_n) \equiv 0 \mod 4,$
it holds $ N_1 + 2N_2 + 3 N_3 \equiv 0 \mod 4.$ Combining these two equations, we obtain $N_2 + N_3 \equiv 1 \mod 4$ that yields $N_1'+ N_3' \equiv 1 \mod 4.$ It only remains to note that
$$\lambda_{\mathbb{Z}_4}(l(y_1), \ldots, l(y_n)) = y_0 \oplus y_1 \oplus \ldots \oplus y_n = N_1' \oplus N_3' =1.$$ 

3. \textit{Necessity.}   
Consider the standardly semilinear quasigroup $h_{\mathbb{Z}_4}$ defined by the Boolean function $\lambda_{\mathbb{Z}_4}$. Every 2-dimensional plane $P$ in the Boolean $n$-dimensional hypercube consists of one vector of weight $k-1$, two  vectors of weights $k$, and one vector of weight $k+1$ for some $1 \leq k  \leq n-1.$ So by the definition of the function $\lambda_{\mathbb{Z}_4}$, it holds $\sum\limits_{z \in P} \lambda(z) \equiv 1 \mod 2.$

To complete the proof we note that all isotopies preserving the property of standard semilinearity are composed of permutations
$$\sigma' = (1023),~\sigma'' = (0132), ~\sigma''' = (2301),$$
that save the parity of the sum of values of a Boolean function over any 2-dimensional plane.

\textit{Sufficiency.} The proof is by induction on $n$. The case $n=2$ is verified directly.

Let $f$ be an $n$-ary standardly semilinear quasigroup and let us denote by $f'$ a quasigroup defined by the equation $f(x_1, \ldots, x_{n-1}, 0) =x_0.$ Then $f'$ is an $(n-1)$-ary standardly semilinear quasigroup with Boolean function $\lambda'$ having odd sum over all 2-dimensional planes.  By the inductive assumption, there exists an isotopy $\Sigma = (\sigma_0, \sigma_1, \ldots, \sigma_{n-1})$ preserving oddity over all 2-dimensional planes that turns the quasigroup $f'$ into the $(n-1)$-ary quasigroup $h_{\mathbb{Z}_4}$ defined by the Boolean function $\lambda_{\mathbb{Z}_4}$ on the $(n-1)$-dimensional Boolean hypercube.

Apply the isotopy $\Sigma$ to the quasigroup $f$ and obtain some quasigroup $g$ with a Boolean function $\lambda''$ such that the sum of values of $\lambda''$ over any 2-dimensional plane is odd. Note that there exists a hyperplane $\Gamma$ of the $n$-dimensional Boolean hypercube in which the function $\lambda''$ is exactly the function $\lambda_{\mathbb{Z}_4}$. Then knowledge of one of values of the function $\lambda''$ outside the hyperplane $\Gamma$ allows us to find values of the function $\lambda''$ on entries which belong to 2-dimensional planes intersecting the hyperplane $\Gamma$ and containing the element outside the hyperplane $\Gamma$, and so it allows us to reconstruct all other values of the function $\lambda''.$  

Thus we have only two possibilities  for the quasigroup $g$, one of which coincides with the $n$-ary quasigroup $h_{\mathbb{Z}_4},$ and another can be turned into this quasigroup by the isotopy $(\varepsilon, \ldots, \varepsilon, (0132))$, where $\varepsilon$ is the identical permutation.
\end{proof}

\textbf{Proof of Lemma~\ref{z22char}}

\begin{proof}
1. By definition, the iterated group $\mathbb{Z}_2^2$ is completely reducible and an application of an isotopy preserves this property.  

2. Consider the $n$-ary iterated group $\mathbb{Z}_2^2$:
$$x_0 + x_1 + \ldots + x_n = 0, ~ x_i \in \mathbb{Z}_2^2.$$

It is known that elements $a \in \mathbb{Z}_2^2$ can be  represented by Boolean vectors $\tilde{a} = (l(a), p(a))$, where  $p(a) \equiv a \mod 2$. Note that for all $a_1, a_2, a_3 \in \mathbb{Z}_2^2$ it holds
$$a_1 + a_2 = a_3 \Leftrightarrow \tilde{a_1} \oplus \tilde{a_2} = \tilde{a_3}.$$

Then for all $x_0, x_1, \ldots, x_n \in \mathbb{Z}_2^2$ we have
$$x_0 +x_1 + \ldots + x_n = 0 \Leftrightarrow l(x_0) \oplus \ldots \oplus l(x_n) = 0 \mbox{ and } x_0 \oplus \ldots \oplus x_n = 0.$$

Therefore the iterated group $\mathbb{Z}_2^2$ coincides with the standardly semilinear quasigroup $h_{\mathbb{Z}_2^2}$ defined by the identical zero Boolean function $\lambda_{\mathbb{Z}_2^2}.$

3. The proof is analogical to the proof of Lemma~\ref{z4char}.
\end{proof}

\textbf{Proof of Lemma~\ref{bool}}

For convenience we reformulate Lemma~\ref{bool} as follows.

\textit{ 
Let $\lambda$ be a Boolean function on the $n$-dimensional Boolean hypercube. Suppose that for every proper quadruple $\left\{z^1, z^2, z^3, z^4\right\}$ of different $n$-vectors, where two of $z^j$ have even weight and other two of $z^j$ have odd weight, it holds $\lambda(z^1) \oplus \lambda(z^2) \oplus \lambda(z^3) \oplus \lambda(z^4) = \delta.$
\begin{enumerate}
\item If $\delta = 1$ and $n$ is even, then for every $2$-dimensional plane $P$ in the Boolean hypercube we have $\sum\limits_{z \in P} \lambda(z) \equiv 1 \mod 2.$ If  $n$ is odd then for any Boolean function $\lambda$ the sum $\lambda(z^1) \oplus \lambda(z^2) \oplus \lambda(z^3) \oplus \lambda(z^4)$ equals to $0$ for at least $\frac{1}{6}(4^{n-1} -2^{n-1})$ such proper quadruples $\left\{z^1, z^2, z^3, z^4\right\}$. 
\item If $\delta = 0$ and $n$ is even, then for every $2$-dimensional plane $P$ in the Boolean hypercube $\sum\limits_{z \in P} \lambda(z) \equiv 0 \mod 2,$ and if $n$ is odd then for every $2$-dimensional planes $P_1$ and $P_2$ it holds $\sum\limits_{z \in P_1} \lambda(z) \equiv \sum\limits_{z \in P_2} \lambda(z) \mod 2.$
\end{enumerate} } 

\begin{proof}
1. Let $n$ be even greater than 2, because for $n=2$ the statement is obviously true. Without loss of generality, we take the 2-dimensional plane $P$ composed by vectors 
$a^1, a^2, a^3,$ and $a^4$, where $a^2$ and $a^4$ have an even weight and $a^1$ and $a^3$ have an odd weight. For example, we may suppose that
$$a^1 = (0, 1, 0, \ldots, 0); ~a^2 = (0, 0, 0, \ldots, 0); $$
$$a^3 = (1, 0, 0, \ldots, 0); ~a^4 = (1, 1, 0, \ldots, 0). $$
Assume that $\lambda(a^1) \oplus \lambda(a^2) \oplus \lambda(a^3) \oplus \lambda(a^4) = 0$.

Consider the quadruple $\left\{a^2, a^4, \nu(a^1), \nu(a^3)\right\}$. Note that it is proper and contains two vectors of even weight and two vectors of odd weight. By the condition of the lemma,
 $\lambda(a^2) \oplus \lambda(a^4) \oplus \lambda(\nu(a^1)) \oplus \lambda(\nu(a^3)) = 1.$  Note that the same is true for quadruples $\left\{a^1, a^2, \nu(a^1), \nu(a^2)\right\}$ and $\left\{a^2, a^3, \nu(a^2), \nu(a^3)\right\}$. Consequently,
\begin{gather*}
\lambda(a^2) \oplus \lambda(a^4) \oplus \lambda(\nu(a^1)) \oplus \lambda(\nu(a^3)) = (\lambda(a^2) \oplus \lambda(a^3) \oplus \lambda(\nu(a^2)) \oplus \lambda(\nu(a^3))) \oplus \\
\oplus (\lambda(a^1) \oplus \lambda(a^2) \oplus \lambda(\nu(a^1)) \oplus \lambda(\nu(a^2))) \oplus (\lambda(a^1) \oplus \lambda(a^2) \oplus \lambda(a^3) \oplus \lambda(a^4))  = \\
=  0 \mbox{ (by assumption) } \oplus 1 \mbox{ (by condition) }  \oplus 1 \mbox{ (by condition) }= 0; 
\end{gather*}
a contradiction.
Therefore the sum of values of $\lambda$ over every 2-dimensional plane is odd.

Suppose $n \geq 3$ is odd. Let us consider the set $U$ of all proper quadruples of the form $\left\{z^1, z^2, \nu(z^1), \nu(z^2)\right\}$, where $z^1 \neq z^2$ and where vectors $z^1$ and $z^2$ has an even weight. The cardinality of the set $U$ is equal to the number of unordered pairs of different even-weight vectors in the $n$-dimensional Boolean hypercube and equals $\frac{1}{2}(4^{n-1} -2^{n-1})$. 

Divide the set $U$ into disjoint subsets $Z_1, \ldots, Z_6$ containing proper quadruples of the following form:
$$Z_1 \ni \left\{z^1, z^2, \nu(z^1), \nu(z^2)\right\};~ Z_2 \ni \left\{z^1, z^3, \nu(z^1), \nu(z^3)\right\};$$
$$Z_3 \ni \left\{z^1, z^4, \nu(z^1), \nu(z^4)\right\};~ Z_4 \ni \left\{z^2, z^3, \nu(z^2), \nu(z^3)\right\};$$
$$Z_5 \ni \left\{z^2, z^4, \nu(z^2), \nu(z^4)\right\};~ Z_6 \ni \left\{z^3, z^4, \nu(z^3), \nu(z^4)\right\}.$$  

Put $\Lambda_i = \lambda(z^k) \oplus \lambda(z^l) \oplus \lambda(\nu(z^k)) \oplus \lambda (\nu(z^l))$ for $\left\{z^k,z^l, \nu(z^k),\nu(z^l)\right\} \in Z_i$. It is easy to see that $\Lambda_i$ satisfy the equalities
$$\Lambda_1 \oplus \Lambda_2 \oplus \Lambda_4 = 0; ~~ \Lambda_1 \oplus \Lambda_3 \oplus \Lambda_5 = 0; ~~\Lambda_2 \oplus \Lambda_3 \oplus \Lambda_6 = 0$$
implying that at least two of six $\Lambda_i$ equals zero. Therefore, the sum of values of a function $\lambda$ is even at least on one third of quadruples from $U$.

2.  The proof for even $n$ is analogical to one given in clause 1. 

Suppose that $n$ is odd greater then 3. For $n=3$ the statement can be verified directly. Let us prove that for every 2-dimensional planes $P_1$ and $P_2$ such that $P_1 \cap P_2 \neq \emptyset$ it holds $\sum\limits_{z \in P_1} \lambda(z) \equiv \sum\limits_{z \in P_2} \lambda(z) \mod 2$. This implies that the same holds for all pairs of 2-dimensional planes.

Without loss of generality, we suppose that plane $P_1$ consists of vectors $a^1, a^2, a^3, a^4$, the plane $P_2$ consists of vectors $a^1, a^2, a^5, a^6$, where $a^2, a^4, a^6$ have an even weight, $a^1, a^3, a^5$ have an odd weight, $\sum\limits_{a^j \in P_1} \lambda(a^j) \equiv 0 \mod 2$, and $\sum\limits_{a^j \in P_2} \lambda(a^j) \equiv 1 \mod 2$. Then
\begin{gather*}
\lambda(a^3) \oplus \lambda(a^4) \oplus \lambda(a^5) \oplus \lambda(a^6) = \\
= (\lambda(a^1) \oplus \lambda(a^2) \oplus \lambda(a^3) \oplus \lambda(a^4)) \oplus (\lambda(a^1) \oplus \lambda(a^2) \oplus \lambda(a^5) \oplus \lambda(a^6)) = \\ = 0 \oplus 1 = 1.
\end{gather*}

Note that quadruples $\left\{a^3, a^4, \nu(a^3), \nu(a^4)\right\}$ and $\left\{a^5, a^6, \nu(a^5), \nu(a^6)\right\}$ are proper and consist of two vectors of even weight and two vectors of odd weight. Therefore,
\begin{gather*}
\lambda(\nu(a^3)) \oplus \lambda(\nu(a^4)) \oplus \lambda(\nu(a^5)) \oplus \lambda(\nu(a^6)) = (\lambda(a^3) \oplus \lambda(a^4) \oplus \lambda(a^5) \oplus \lambda(a^6)) \oplus \\
\oplus (\lambda(a^3) \oplus \lambda(a^4) \oplus \lambda(\nu(a^3)) \oplus \lambda(\nu(a^4))) \oplus (\lambda(a^5) \oplus \lambda(a^6) \oplus \lambda(\nu(a^5)) \oplus \lambda(\nu(a^6))) \\ = 1 \oplus 0 \oplus 0 = 1.
\end{gather*}

Since $\lambda(a^3) \oplus \lambda(a^4) \oplus \lambda(a^5) \oplus \lambda(a^6) = 1$,  we have $\lambda(a^3) \oplus \lambda(a^6) = 1$ or $\lambda(a^4) \oplus \lambda(a^5) = 1$. Similarly, $\lambda(\nu(a^3)) \oplus \lambda(\nu(a^4)) \oplus \lambda(\nu(a^5)) \oplus \lambda(\nu(a^6)) = 1$ implies that $\lambda(\nu(a^3)) \oplus \lambda(\nu(a^6)) = 0$ or $\lambda(\nu(a^4)) \oplus \lambda(\nu(a^5)) = 0.$ It is easy to see that quadruples $\left\{a^3, a^6, \nu(a^3), \nu(a^6)\right\}$, $\left\{a^3, a^6, \nu(a^4), \nu(a^5)\right\}$, $\left\{a^4, a^5, \nu(a^3), \nu(a^6)\right\}$, and $\left\{a^4, a^5, \nu(a^4), \nu(a^5)\right\}$ are proper and contains two vectors of even weight and two vectors of odd weight. By the condition of the lemma the function $\lambda$ must have an even sum on all of them, that is impossible.
\end{proof}

\textbf{Proof of Lemma~\ref{almostz4}}

\begin{proof}

We prove the following equivalent statement.

\textit{Let $f$ be an $n$-ary quasigroup of order $4$.
Suppose for all $a \in \Sigma_4$  the quasigroups $f^a$ defined by the equation $f(x_1, \ldots, x_n) = a$ are $\mathbb{Z}_4$-linear. Then the quasigroup $f$ is completely reducible.}

Indeed, if all  quasigroups $f^a$ are $\mathbb{Z}_4$-linear and if $f$ is completely reducible then the quasigroup $f$ is a composition of an $(n-1)$-ary $\mathbb{Z}_4$-linear quasigroup and some binary quasigroup.

The proof is by induction on $n$. 
For $n \leq 5$ the statement is verified directly with the help of the list of latin hypercubes of small orders and dimensions provided by~\cite{wancec}.

For the inductive step we will use the following result obtained in~\cite{krpt}.

\begin{utv}\label{compredvsp}
Let $n \geq 5$ and let $f$ be an $n$-ary quasigroup of order $4$. Assume that for all $k \in \left\{1, \ldots, n-2 \right\}$ and for all choices of $0 \leq i_1 < \ldots < i_{k} \leq n $ and $a_{i_1}, \ldots, a_{i_k} \in \Sigma_4$ quasigroups $f_{a_{i_1}, \ldots, a_{i_k}}^{i_1, \ldots, i_k}$  defined by relations
$$f(x_1, \ldots, x_n) = x_0,~ x_{i_1} = a_{i_1}, \ldots, ~ x_{i_k} = a_{i_{k}}$$
are reducible. Then the quasigroup $f$ is completely reducible.
\end{utv}

Suppose $n \geq 5$ and $f$ is an $n$-ary quasigroup of order 4 such that for all $a \in \Sigma_4$  the quasigroups $f_a^0$ are $\mathbb{Z}_4$-linear. Let us choose arbitrary $k \in \left\{1, \ldots, n-2 \right\}$, numbers  $0 \leq i_1 < \ldots  < i_{k}  \leq n $, and elements $a_{i_1}, \ldots, a_{i_k} \in \Sigma_4$ and consider the quasigroup $f_{a_{i_1}, \ldots, a_{i_k}}^{i_1, \ldots, i_k}$. 

If $i_1 = 0$ then the quasigroup $f_{a_{i_1}, \ldots, a_{i_k}}^{i_1, \ldots, i_k}$ is obtained from the $\mathbb{Z}_4$-linear quasigroup $f_{a_{i_0}}^0$ by fixing values of some indeterminates and so it is reducible.
If $i_1 \neq 0$ then  for all $a \in \Sigma_4$ consider the quasigroups $f_{a, a_{i_1}, \ldots, a_{i_k}}^{0, i_1, \ldots, i_k}$.
 
The quasigroup $f_{a, a_{i_1}, \ldots, a_{i_k}}^{0, i_1, \ldots, i_k}$ is obtained from the $\mathbb{Z}_4$-linear quasigroup $f_{a}^0$ by fixing values of some indeterminates and so it is $\mathbb{Z}_4$-linear. By the inductive assumption, the quasigroup $f_{a_{i_1}, \ldots, a_{i_k}}^{i_1, \ldots, i_k}$ is completely reducible.

Since for all $k \in \left\{1, \ldots, n-2 \right\}$ and for all selections $0 \leq i_1 < \ldots < i_{k} \leq n $ and $a_{i_1}, \ldots, a_{i_k} \in \Sigma_4$ quasigroups $f_{a_{i_1}, \ldots, a_{i_k}}^{i_1, \ldots, i_k}$ are reducible, then by Proposition~\ref{compredvsp}, the quasigroup $f$ is completely reducible.
\end{proof}

\textbf{Proof of Lemma~\ref{numbrind}}

\begin{proof}
Let $\mathcal{A}_n$ be a set of $4 \times (n+1)$ Boolean matrices in which each column contains exactly 2 zeroes and 2 ones. We introduce the following notation:
\begin{itemize}
\item $A^{00}(n)$ is the number of matrices $A \in \mathcal{A}_n$ such that the sum of entries over each row is even.
\item $A^{01}(n)$ is the number of matrices $A \in \mathcal{A}_n$ such that the sum of entries over two rows is even and the sum over other two rows is odd.
\item $A^{11}(n)$ is the number of matrices $A \in \mathcal{A}_n$ such that the sum of entries over each row is odd.
\item $B^{00}(n)$ is the number of matrices $A \in \mathcal{A}_n$ with two pairs of identical rows such that the sum of entries over each row is even.
\item $B^{01}(n)$ is the number of matrices $A \in \mathcal{A}_n$ with two pairs of identical rows such that the sum of entries over two rows is even and the sum over other two rows is odd.
\item $B^{11}(n)$ is the number of matrices $A \in \mathcal{A}_n$ with two pairs of identical rows such that the sum of entries over each row is odd.
\end{itemize}

Note that the number of brindled quadruples in the $(n+1)$-dimensional Boolean hypercube is expressed as
$$W(n) = \frac{1}{4!} \left(A^{00} (n) - B^{00} (n)\right).$$

The numbers $A^{00}(n), A^{01}(n)$, and $A^{11}(n)$ satisfy the following recurrence:
$$A^{00}(n) = A^{01} (n-1); ~ A^{11}(n) = A^{01}(n-1)$$ 
$$A^{01} (n) = 4 A^{01}(n-1) + 6(A^{00}(n-1) + A^{11}(n-1)).$$
Therefore we have
$$A^{01}(n) = 4 A^{01} (n-1) + 12 A^{01}(n-2).$$
Solving this recurrence with $A^{01} (0) = 6$ and $A^{01}(1) = 24$, we obtain
$$A^{01}(n) = \frac{3}{4} \left(6^{n+1} - (-2)^{n+1}\right).$$
Consequently,
$$A^{00}(n) = \frac{3}{4} \left(6^{n} - (-2)^{n}\right).$$

Next we note that $B^{00}(n)$ and $B^{11}(n)$ are equal to zero if $n$ is even, and $B^{01}(n)$ is zero if $n$ is odd.  Also, $B^{00}(n), B^{01}(n)$, and $B^{11}(n)$ satisfy the following relations:
$$B^{00}(2k+1) = B^{01}(2k); ~ B^{11}(2k+1) = B^{01}(2k); $$
$$B^{01}(2k+2) = 2(B^{00}(2k+1) + B^{11}(2k+1)).$$
Then for $B^{01}(2k)$ it holds
$$B^{01}(2k) = 4B^{01}(2k-2).$$
Solving this recurrence with $B^{01}(0) = 6$, we obtain
$$B^{01}(2k) = 6 \cdot 2^{2k}.$$
Consequently $B^{00}(n) = 0$ for even $n$ and $B^{00}(n) = 3 \cdot 2^{n}$ for odd $n$.

Finally, if $n$ is even then
$$W(n) = \frac{1}{24} \left(A^{00} (n) - B^{00} (n)\right) = \frac{1}{32} \left(6^n - 2^n\right),$$
and if $n$ is odd then
$$W(n) = \frac{1}{24} \left(A^{00} (n) - B^{00} (n)\right) = \frac{1}{32} \left(6^n + 2^n - 4 \cdot 2^n \right) = \frac{1}{32} \left(6^n - 3 \cdot 2^n \right).$$
\end{proof}

\section*{Acknowledgments}
The author is grateful to V.N. Potapov for his help and for constant attention to this work.
The author was supported by the Russian Science Foundation (grant 14--11--00555) (Sections 3--7) and the Moebius Contest Foundation for Young Scientists (Sections 1--2).

\end{document}